\documentclass[10pt,twoside]{article}
\usepackage{mathrsfs}
\usepackage{amssymb}
\usepackage{amsmath}
\usepackage[all]{xy}
\usepackage{amsthm}

\usepackage{url}
\numberwithin{equation}{section}

\newtheorem{theorem}{Theorem}[section]
\newtheorem{proposition}[theorem]{Proposition}
\newtheorem{definition}[theorem]{Definition}

\newtheorem{lemma}[theorem]{Lemma}

\newtheorem{corollary}[theorem]{Corollary}

\newtheorem{theorem*}{Theorem}

\newcommand{\Tor}{\operatorname{Tor}}
\newcommand{\Mod}{\operatorname{Mod}}
\newcommand{\Hom}{\operatorname{Hom}}
\newcommand{\Ext}{\operatorname{Ext}}

\newcommand{\ra}{\rightarrow}

\def\max{\mathop{\rm max}\nolimits}

\def\Ker{\mathop{\rm Ker}\nolimits}
\def\Im{\mathop{\rm Im}\nolimits}
\def\Coker{\mathop{\rm Coker}\nolimits}

\def\pd{\mathop{\rm pd}\nolimits}
\def\fd{\mathop{\rm fd}\nolimits}

\def\Mod{\mathop{\rm Mod}\nolimits}

\def\Prod{\mathop{\rm Prod}\nolimits}

\title{ \bf Duality Pairs Induced by Auslander and Bass Classes
}
\vspace{0.2cm}

\author{Zhaoyong Huang\thanks{{\it E-mail address}: huangzy@nju.edu.cn}\\
{\footnotesize \it Department of Mathematics, Nanjing University, Nanjing 210093, Jiangsu Province, P.R. China}}
\date{ }
\begin{document}

\baselineskip=16pt
\maketitle

\begin{abstract}
Let $R$ and $S$ be any rings and $_RC_S$ a semidualizing bimodule, and let $\mathcal{A}_C(R^{op})$
and $\mathcal{B}_C(R)$ be the Auslander and Bass classes respectively. Then both the pairs
$$(\mathcal{A}_C(R^{op}),\mathcal{B}_C(R))\ {\rm and}\ (\mathcal{B}_C(R),\mathcal{A}_C(R^{op}))$$
are coproduct-closed and product-closed duality pairs and both
$\mathcal{A}_C(R^{op})$ and $\mathcal{B}_C(R)$ are covering and preenveloping; in particular,
the former duality pair is perfect. Moreover,
if $\mathcal{B}_C(R)$ is enveloping in $\Mod R$, then $\mathcal{A}_C(S)$ is enveloping in $\Mod S$.
Then some applications to the Auslander projective dimension of modules are given.
\end{abstract}

\pagestyle{myheadings}
\markboth{\rightline {\scriptsize Z. Y. Huang}}
         {\leftline{\scriptsize  Duality Pairs Induced by Auslander and Bass Classes}}

\section{Introduction} 

In relative homological algebra, the theory of covers and envelopes is fundamental and important.
Let $R$ be a ring and $\Mod R$ the category of left $R$-modules. Given a subcategory of $\Mod R$,
it is always worth studying whether or when it is (pre)covering or (pre)enveloping.
This problem has been studied extensively, see \cite{BR}--\cite{HJ09} and references therein.

Let $R$ be a commutative noetherian ring and $C$ a semidualizing $R$-module,
and let $\mathcal{A}_C(R)$ and $\mathcal{B}_C(R)$ be the Auslander and Bass classes respectively.
By proving that both $\mathcal{A}_C(R)$ and $\mathcal{B}_C(R)$ are Kaplansky classes, Enochs and Holm
got in \cite[Theorems 3.11 and 3.12]{EH} that the pair $(\mathcal{A}_C(R),(\mathcal{A}_C(R))^\bot)$
is a perfect cotorsion pair, $\mathcal{A}_C(R)$ is covering and preenveloping
and $\mathcal{B}_C(R)$ is preenveloping. Holm and J{\o}rgensen introduced the notion of duality pairs
and proved the following remarkable result. Let $R$ be an arbitrary ring, and let $\mathscr{X}$ and
$\mathscr{Y}$ be subcategories of $\Mod R$ and $\Mod R^{op}$ respectively.
When $(\mathscr{X},\mathscr{Y})$ is a duality pair, the following assertions hold true:
(1) If $\mathscr{X}$ is closed under coproducts, then $\mathscr{X}$ is covering;
(2) if $\mathscr{X}$ is closed under products, then $\mathscr{X}$ is preenveloping; and
(3) if $_RR\in\mathscr{X}$ and $\mathscr{X}$ is closed under coproducts and extensions,
then $(\mathscr{X},\mathscr{X}^{\perp})$ is a perfect cotorsion pair
(\cite[Theorem 3.1]{HJ09}). By using it, they generalized the above result of Enochs and Holm to
the category of complexes, and Enochs and Iacob investigated in \cite{EI} the existence of Gorenstein
injective envelopes over commutative noetherian rings.

Let $R$ and $S$ be arbitrary rings and $_RC_S$ a semidualizing bimodule, and let $\mathcal{A}_C(R^{op})$
be the Auslander class in $\Mod R^{op}$ and $\mathcal{B}_C(R)$ the Bass class in $\Mod R$.
Our first main result is the following

\begin{theorem}\label{1.1} {\rm (Theorem \ref{3.3})}
\begin{enumerate}
\item[(1)] Both the pairs
$$(\mathcal{A}_C(R^{op}),\mathcal{B}_C(R))\ and\ (\mathcal{B}_C(R),\mathcal{A}_C(R^{op}))$$
are coproduct-closed and product-closed duality pairs; and furthermore, the former one is perfect.
\item[(2)] $\mathcal{A}_C(R^{op})$ is covering and preenveloping in $\Mod R^{op}$
and $\mathcal{B}_C(R)$ is covering and preenveloping in $\Mod R$.
\end{enumerate}
\end{theorem}

As a consequence of Theorem \ref{1.1}, we get that the pair
$$(\mathcal{A}_C(R^{op}),\mathcal{A}_C(R^{op})^{\bot})$$
is a hereditary perfect cotorsion pair and $\mathcal{A}_C(R^{op})$ is covering and preenveloping in $\Mod R^{op}$,
where $\mathcal{A}_C(R^{op})^{\bot}$ is the right $\Ext$-orthogonal class of $\mathcal{A}_C(R^{op})$ (Corollary \ref{3.4}).
This result was proved in \cite[Theorem 3.11]{EH} when $R$ is a commutative noetherian ring and $_RC_S={_RC_R}$.

By Theorem \ref{1.1} and its symmetric result, we have that $\mathcal{B}_C(R)$ is preenveloping in $\Mod R$
and $\mathcal{A}_C(S)$ is preenveloping in $\Mod S$. Moreover, we prove the following

\begin{theorem}\label{1.2} {\rm (Theorem \ref{3.7}(2))}
If $\mathcal{B}_C(R)$ is enveloping in $\Mod R$, then $\mathcal{A}_C(S)$ is enveloping in $\Mod S$.
\end{theorem}

Then we apply these results and their symmetric results to study the Auslander projective dimension of modules.
We obtain some criteria for computing the Auslander projective dimension of modules in $\Mod S$ (Theorem \ref{4.4}).
Furthermore, we get the following

\begin{theorem}\label{1.3} {\rm (Theorem \ref{4.10})}
If $_RC$ has an ultimately closed projective resolution, then
$$\mathcal{A}_C(S)={{C_S}^{\top}}={^{\bot}\mathcal{I}_C(S)},$$
where ${{C_S}^{\top}}$ is the $\Tor$-orthogonal class of $C_S$ and ${^{\bot}\mathcal{I}_C(S)}$ is
the left $\Ext$-orthogonal class of the subcategory $\mathcal{I}_C(S)$ of $\Mod S$ consisting of $C$-injective modules.
\end{theorem}

As a consequence, we have that if $_RC$ has an ultimately closed projective resolution,
then the projective dimension of $C_S$ is at most $n$ if and only if the Auslander projective dimension of
any module in $\Mod S$ is at most $n$ (Corollary \ref{4.11}).

\section{Preliminaries}

In this paper, all rings are associative with identities. Let $R$ be a ring. We use $\Mod R$ to denote the category
of left $R$-modules and all subcategories of $\Mod R$ are full and closed under isomorphisms.
For a subcategory $\mathscr{X}$ of $\Mod R$, we write
$${^\perp{\mathscr{X}}}:=\{A\in\Mod R\mid\operatorname{Ext}^{\geq 1}_{R}(A,X)=0 \mbox{ for any}\ X\in \mathscr{X}\},$$
$${{\mathscr{X}}^\perp}:=\{A\in\Mod R\mid\operatorname{Ext}^{\geq 1}_{R}(X,A)=0 \mbox{ for any}\ X\in \mathscr{X}\},$$
$${^{\perp_1}{\mathscr{X}}}:=\{A\in\Mod R\mid\operatorname{Ext}^{1}_{R}(A,X)=0 \mbox{ for any}\ X\in \mathscr{X}\},$$
$${{\mathscr{X}}^{\perp_1}}:=\{A\in\Mod R\mid\operatorname{Ext}^{1}_{R}(X,A)=0 \mbox{ for any}\ X\in \mathscr{X}\}.$$
For subcategories $\mathscr{X},\mathscr{Y}$ of $\Mod R$, we write $\mathscr{X}\perp\mathscr{Y}$ if
$\operatorname{Ext}^{\geq 1}_{R}(X,Y)=0$ for any $X\in \mathscr{X}$ and $Y\in \mathscr{Y}$.

\begin{definition} \label{2.1}
{\rm (\cite{E1,EJ00})
Let $\mathscr{X}\subseteq\mathscr{Y}$ be
subcategories of $\Mod R$. A homomorphism $f: X\to Y$ in
$\Mod R$ with $X\in\mathscr{X}$ and $Y\in \mathscr{Y}$ is called an {\bf $\mathscr{X}$-precover} of $Y$
if $\Hom_{R}(X^{'},f)$ is epic for any $X^{'}\in\mathscr{X}$; and $f$ is called {\bf right minimal}
if an endomorphism $h:X\to X$ is an automorphism whenever $f=fh$.
An {\bf $\mathscr{X}$-precover} $f: X\to Y$ is called an {\bf $\mathscr{X}$-cover} of $Y$
if it is right minimal. The subcategory $\mathscr{X}$ is called {\bf (pre)covering}
in $\mathscr{Y}$ if any object in $\mathscr{Y}$ admits an $\mathscr{X}$-(pre)cover.
Dually, the notions of an {\bf $\mathscr{X}$-(pre)envelope}, a {\bf left minimal homomorphism}
and a {\bf (pre)enveloping subcategory} are defined.}
\end{definition}

\begin{definition}\label{2.2}
{\rm (\cite{EJ00, GT12})
Let $\mathscr{U},\mathscr{V}$ be subcategories of $\Mod R$.
\begin{enumerate}
\item[(1)] The pair $(\mathscr{U},\mathscr{V})$ is called a {\bf cotorsion pair} in $\Mod R$ if
$\mathscr{U}={^{\bot_1}\mathscr{V}}$ and $\mathscr{V}={\mathscr{U}^{\bot_1}}$.
\item[(2)] A cotorsion pair $(\mathscr{U},\mathscr{V})$ is called {\bf perfect} if $\mathscr{U}$ is covering and $\mathscr{V}$
is enveloping in $\Mod R$.
\item[(3)] A cotorsion pair $(\mathscr{U},\mathscr{V})$ is called {\bf hereditary} if one of the following equivalent
conditions is satisfied.
\begin{enumerate}
\item[(3.1)] $\mathscr{U}\perp \mathscr{V}$.
\item[(3.2)] $\mathscr{U}$ is projectively resolving in the sense that $\mathscr{U}$ contains all projective modules
in $\Mod R$, $\mathscr{U}$ is closed under extensions and kernels of epimorphisms.
\item[(3.3)] $\mathscr{V}$ is injectively coresolving in the sense that $\mathscr{V}$ contains all injective modules
in $\Mod R$, $\mathscr{V}$ is closed under extensions and cokernels of monomorphisms.
\end{enumerate}
\end{enumerate}}
\end{definition}

Set $(-)^+:=\Hom_{\mathbb{Z}}(-,\mathbb{Q}/\mathbb{Z})$, where $\mathbb{Z}$ is the additive group of integers and $\mathbb{Q}$
is the additive group of rational numbers. The following is the definition of duality pairs (cf. \cite{EI,HJ09}).

\begin{definition}\label{2.3}
{\rm Let $\mathscr{X}$ and $\mathscr{Y}$ be subcategories of $\Mod R$ and $\Mod R^{op}$ respectively.
\begin{enumerate}
\item[(1)] The pair ($\mathscr{X},\mathscr{Y}$) is called a {\bf duality pair} if the following conditions are satisfied.
\begin{enumerate}
\item[(1.1)] For a module $X\in\Mod R$, $X\in\mathscr{X}$  if and only if $X^{+}\in \mathscr{Y}$.
\item[(1.2)] $\mathscr{Y}$ is closed under direct summands and finite direct sums.
\end{enumerate}
\item[(2)] A duality pair ($\mathscr{X},\mathscr{Y}$) is called {\bf (co)product-closed} if $\mathscr{X}$ is closed under
(co)products.
\item[(3)] A duality pair ($\mathscr{X},\mathscr{Y}$) is called {\bf perfect} if it is coproduct-closed,
$_RR\in\mathscr{X}$ and $\mathscr{X}$ is closed under extensions.
\end{enumerate}}
\end{definition}

We also recall the following remarkable result.

\begin{lemma}\label{2.4}
{\rm (\cite[p.7, Theorem]{EI} and \cite[Theorem 3.1]{HJ09})}
Let $\mathscr{X}$ and $\mathscr{Y}$ be subcategories of $\Mod R$ and $\Mod R^{op}$ respectively.
If $(\mathscr{X},\mathscr{Y})$ is a duality pair, then the following assertions hold true.
\begin{enumerate}
\item[(1)] If $(\mathscr{X},\mathscr{Y})$ is coproduct-closed, then $\mathscr{X}$ is covering.
\item[(2)] If $(\mathscr{X},\mathscr{Y})$ is product-closed, then $\mathscr{X}$ is preenveloping.
\item[(3)] If $(\mathscr{X},\mathscr{Y})$ is perfect, then $(\mathscr{X},\mathscr{X}^{\perp})$ is a perfect cotorsion pair.
\end{enumerate}
\end{lemma}

\begin{definition} \label{2.5}
{\rm (\cite{HW07}).
Let $R$ and $S$ be rings. An ($R,S$)-bimodule $_RC_S$ is called
{\bf semidualizing} if the following conditions are satisfied.
\begin{enumerate}
\item[(a1)] $_RC$ admits a degreewise finite $R$-projective resolution.
\item[(a2)] $C_S$ admits a degreewise finite $S$-projective resolution.
\item[(b1)] The homothety map $_RR_R\stackrel{_R\gamma}{\rightarrow} \Hom_{S^{op}}(C,C)$ is an isomorphism.
\item[(b2)] The homothety map $_SS_S\stackrel{\gamma_S}{\rightarrow} \Hom_{R}(C,C)$ is an isomorphism.
\item[(c1)] $\Ext_{R}^{\geq 1}(C,C)=0$.
\item[(c2)] $\Ext_{S^{op}}^{\geq 1}(C,C)=0$.
\end{enumerate}}
\end{definition}

Wakamatsu in \cite{W1} introduced and studied the so-called {\bf generalized tilting modules},
which are usually called {\bf Wakamatsu tilting modules}, see \cite{BR, MR}. Note that
a bimodule $_RC_S$ is semidualizing if and only if it is Wakamatsu tilting (\cite[Corollary 3.2]{W3}).
Examples of semidualizing bimodules are referred to \cite{HW07,W2}.

\section{Duality pairs}

In this section, $R$ and $S$ are arbitrary rings and $_RC_S$ is a semidualizing bimodule.
We write $(-)_*:=\Hom(C,-)$ and
$${{_RC}^{\bot}}:=\{M\in\Mod R\mid\Ext^{\geq 1}_{R}(C,M)=0\}\ \text{and}\
{{C_S}^{\bot}}:=\{B\in\Mod S^{op}\mid\Ext^{\geq 1}_{S^{op}}(C,B)=0\},$$
$${^{\top}{_RC}}:=\{N\in\Mod R^{op}\mid\Tor_{\geq 1}^{R}(N,C)=0\}\ \text{and}\
{{C_S}^{\top}}:=\{A\in\Mod S\mid\Tor_{\geq 1}^{S}(C,A)=0\}.$$

\begin{definition} \label{3.1}
{\rm (\cite{HW07})
\begin{enumerate}
\item[(1)] The {\bf Auslander class} $\mathcal{A}_{C}(R^{op})$ with respect to $C$ consists of all modules $N$
in $\Mod R^{op}$ satisfying the following conditions.
\begin{enumerate}
\item[(a1)] $N\in{^{\top}{_RC}}$.
\item[(a2)] $N\otimes _{R}C\in{{C_S}^{\perp}}$.
\item[(a3)] The canonical valuation homomorphism
$$\mu_N:N\rightarrow (N\otimes_RC)_*$$
defined by $\mu_N(x)(c)=x\otimes c$ for any $x\in N$ and $c\in C$ is an isomorphism in $\Mod R^{op}$.
\end{enumerate}
\item[(2)] The {\bf Bass class} $\mathcal{B}_C(R)$ with respect to $C$ consists of all modules $M$
in $\Mod R$ satisfying the following conditions.
\begin{enumerate}
\item[(b1)] $M\in{_RC^{\perp}}$.
\item[(b2)] $M_*\in{{C_S}^{\top}}$.
\item[(b3)] The canonical valuation homomorphism
$$\theta_M:C\otimes_SM_*\rightarrow M$$
defined by $\theta_M(c\otimes f)=f(c)$ for any $c\in C$ and $f\in M_*$ is an isomorphism in $\Mod R$.
\end{enumerate}
\item[(3)] The {\bf Auslander class} $\mathcal{A}_C(S)$ in $\Mod S$ and the {\bf Bass class} $\mathcal{B}_C(S^{op})$
in $\Mod S^{op}$ are defined symmetrically.
\end{enumerate}}
\end{definition}

The following result is crucial. From its proof, it is known that the conditions in the definitions
of $\mathcal{A}_{C}(R^{op})$ and $\mathcal{B}_C(R)$ are dual item by item.

\begin{proposition}\label{3.2}
\begin{enumerate}
\item[]
\item[(1)] For a module $N\in\Mod R^{op}$, $N\in\mathcal{A}_{C}(R^{op})$ if and only if $N^+\in\mathcal{B}_C(R)$.
\item[(2)] For a module $M\in\Mod R$, $M\in\mathcal{B}_C(R)$ if and only if $M^+\in\mathcal{A}_{C}(R^{op})$.
\end{enumerate}
\end{proposition}

\begin{proof}
(1) Let $N\in \Mod R^{op}$. Then we have the following

(a) \begin{align*}
&\ \ \ \ \ \  \ \ \ \ N\in{^{\top}{_RC}}\\
&\ \ \ \ \ \Leftrightarrow \Tor_{\geq 1}^R(N,C)=0\\
&\ \ \ \ \ \Leftrightarrow [\Tor_{\geq 1}^R(N,C)]^+=0\\
&\ \ \ \ \ \Leftrightarrow \Ext^{\geq 1}_R(C,N^+)=0\ \text{(by \cite[Lemma 2.16(b)]{GT12})}\\
&\ \ \ \ \ \Leftrightarrow N^+\in{_RC^{\perp}}.
\end{align*}

(b) \begin{align*}
&\ \ \ \ \ \  \ \ \ \ N\otimes _{R}C\in{{C_S}^{\perp}}\\
&\ \ \ \ \ \Leftrightarrow \Ext_{S^{op}}^{\geq 1}(C,N\otimes_RC)=0\\
&\ \ \ \ \ \Leftrightarrow [\Ext_{S^{op}}^{\geq 1}(C,N\otimes_RC)]^+=0\\
&\ \ \ \ \ \Leftrightarrow \Tor_{\geq 1}^S(C,(N\otimes_RC)^+)=0\ \text{(by \cite[Lemma 2.16(d)]{GT12})}\\
&\ \ \ \ \ \Leftrightarrow \Tor_{\geq 1}^S(C,(N^+)_*)=0\ \text{(by \cite[Lemma 2.16(a)]{GT12})}\\
&\ \ \ \ \ \Leftrightarrow (N^+)_*\in{{C_S}^{\top}}.
\end{align*}

(c) By \cite[Lemma 2.16(c)]{GT12}, the canonical valuation homomorphism
$$\alpha:C\otimes_S(N\otimes_RC)^+ \to [\Hom_{S^{op}}(C,N\otimes_RC)]^+$$
defined by $\alpha(c\otimes g)(f)=gf(c)$ for any $c\in C$, $g\in (N\otimes_RC)^+$ and $f\in\Hom_{S^{op}}(C,N\otimes_RC)$
is an isomorphism in $\Mod R$. By \cite[Lemma 2.16(a)]{GT12}, the canonical valuation homomorphism
$$\beta:(N\otimes_RC)^+\to \Hom_R(C,N^+)$$
defined by $\beta(g)(c)(x)=g(x\otimes c)$ for any $g\in (N\otimes_RC)^+$, $c\in C$ and $x\in N$
is an isomorphism in $\Mod S$. So $$1_C\otimes\beta:C\otimes_S(N\otimes_RC)^+\to C\otimes_S\Hom_R(C,N^+)$$
via $(1_C\otimes\beta)(c\otimes g)=c\otimes\beta(g)$ for any $c\in C$ and $g\in (N\otimes_RC)^+$ is
an isomorphism in $\Mod R$.

Consider the following diagram
\begin{gather*}
\begin{split}
\xymatrix{
& C\otimes_S(N\otimes_RC)^+ \ar[rr]^{\alpha} \ar [d]^{1_C\otimes\beta} && [\Hom_{S^{op}}(C,N\otimes_RC)]^+ \ar [d]^{(\mu_N)^+} \\
& C\otimes_S\Hom_R(C,N^+) \ar[rr]^{\theta_{N^+}} &&N^+,}
\end{split}
\end{gather*}
where
$$(\mu_N)^+:[\Hom_{S^{op}}(C,N\otimes_RC)]^+\to N^+$$
via $(\mu_N)^+(f^{'})=f^{'}\mu_N$ for any $f^{'}\in [\Hom_{S^{op}}(C,N\otimes_RC)]^+$ is a natural homomorphism in $\Mod R$,
and
$$\theta_{N^+}:C\otimes_S\Hom_R(C,N^+)\to N^+$$
defined by $\theta_{N^+}(c\otimes f^{''})=f^{''}(c)$ for any $c\in C$ and $f^{''}\in\Hom_R(C,N^+)$
is a canonical valuation homomorphism in $\Mod R$. Then for any $c\in C$, $g\in(N\otimes_RC)^+$ and $x\in N$, we have
$$(\mu_N)^+\alpha(c\otimes g)(x)=\alpha(c\otimes g)\mu_N(x)=g\mu_N(x)(c)=g(x\otimes c)$$
$$\theta_{N^+}(1_C\otimes\beta)(c\otimes g)(x)=\theta_{N^+}(c\otimes \beta(g))(x)=\beta(g)(c)(x)=g(x\otimes c),$$
Thus $$(\mu_N)^+\alpha=\theta_{N^+}(1_C\otimes\beta),$$
and therefore $\mu_N$ is an isomorphism $\Leftrightarrow$ $(\mu_N)^+$ is an isomorphism
$\Leftrightarrow$ $\theta_{N^+}$ is an isomorphism.

We conclude that $N\in\mathcal{A}_{C}(R^{op})\Leftrightarrow N^+\in\mathcal{B}_C(R)$.

(2) Let $M\in \Mod R$. Then we have the following

(a) \begin{align*}
&\ \ \ \ \ \  \ \ \ \ M\in{_RC^{\perp}}\\
&\ \ \ \ \ \Leftrightarrow \Ext^{\geq 1}_R(C,M)=0\\
&\ \ \ \ \ \Leftrightarrow [\Ext^{\geq 1}_R(C,M)]^+=0\\
&\ \ \ \ \ \Leftrightarrow \Tor_{\geq 1}^R(M^+,C)=0\ \text{(by \cite[Lemma 2.16(d)]{GT12})}\\
&\ \ \ \ \ \Leftrightarrow M^+\in{^{\top}{_RC}}.
\end{align*}

(b) \begin{align*}
&\ \ \ \ \ \  \ \ \ \ M_*\in{{C_S}^{\top}}\\
&\ \ \ \ \ \Leftrightarrow \Tor_{\geq 1}^S(C,M_*)=0\\
&\ \ \ \ \ \Leftrightarrow [\Tor_{\geq 1}^S(C,M_*)]^+=0\\
&\ \ \ \ \ \Leftrightarrow \Ext_{S^{op}}^{\geq 1}(C,(M_*)^+)=0\ \text{(by \cite[Lemma 2.16(b)]{GT12})}\\
&\ \ \ \ \ \Leftrightarrow \Ext_{S^{op}}^{\geq 1}(C,M^+\otimes_RC)=0\ \text{(by \cite[Lemma 2.16(c)]{GT12})}\\
&\ \ \ \ \ \Leftrightarrow M^+\otimes_RC\in{{C_S}^{\perp}}.
\end{align*}

(c) By \cite[Lemma 2.16(a)]{GT12}, the canonical valuation homomorphism
$$\tau:[C\otimes_S\Hom_R(C,M)]^+\to \Hom_{S^{op}}(C,[\Hom_R(C,M)]^+)$$
defined by $\tau(g^{'})(c)(f)=g^{'}(c\otimes f)$ for any $g^{'}\in [C\otimes_S\Hom_R(C,M)]^+$, $c\in C$ and $f\in\Hom_R(C,M)$
is an isomorphism in $\Mod R^{op}$. By \cite[Lemma 2.16(c)]{GT12}, the canonical valuation homomorphism
$$\sigma:M^+\otimes_RC\to [\Hom_R(C,M)]^+$$
defined by $\sigma(g\otimes c)(f)=gf(c)$ for any $g\in M^+$, $c\in C$ and $f\in\Hom_R(C,M)$
is an isomorphism in $\Mod S^{op}$. So $$\Hom_{S^{op}}(C,\sigma):\Hom_{S^{op}}(C,M^+\otimes_RC)\to \Hom_{S^{op}}(C,[\Hom_R(C,M)]^+)$$
via $\Hom_{S^{op}}(C,\sigma)(g^{''})=\sigma g^{''}$ for any $g^{''}\in \Hom_{S^{op}}(C,M^+\otimes_RC)$ is
an isomorphism in $\Mod R^{op}$.

Consider the following diagram
\begin{gather*}
\begin{split}
\xymatrix{
& M^+ \ar[rr]^{(\theta_M)^+} \ar [d]^{\mu_{M^+}} && [C\otimes_S\Hom_R(C,M)]^+ \ar [d]^{\tau} \\
& \Hom_{S^{op}}(C,M^+\otimes_RC) \ar[rr]^{\Hom_{S^{op}}(C,\sigma)} &&\Hom_{S^{op}}(C,[\Hom_R(C,M)]^+),}
\end{split}
\end{gather*}
where
$$(\theta_M)^+:M^+\to [C\otimes_S\Hom_R(C,M)]^+$$
via $(\theta_M)^+(g)=g\theta_M$ for any $g\in M^+$ is a natural homomorphism in $\Mod R^{op}$, and
$$\mu_{M^+}:M^+\to \Hom_{S^{op}}(C,M^+\otimes_RC)$$ defined by $\mu_{M^+}(g)(c)=g\otimes c$ for any $g\in M^+$ and $c\in C$
is a canonical valuation homomorphism in $\Mod R^{op}$. Then for any $g\in M^+$, $c\in C$ and $f\in\Hom_R(C,M)$, we have
$$\tau(\theta_M)^+(g)(c)(f)=(\theta_M)^+(g)(c\otimes f)=g\theta_M(c\otimes f)=gf(c),$$
$$\Hom_{S^{op}}(C,\sigma)\mu_{M^+}(g)(c)(f)=\sigma\mu_{M^+}(g)(c)(f)=\sigma(g\otimes c)(f)=gf(c),$$
Thus $$\tau(\theta_M)^+=\Hom_{S^{op}}(C,\sigma)\mu_{M^+},$$
and therefore $\theta_M$ is an isomorphism $\Leftrightarrow$ $(\theta_M)^+$ is an isomorphism
$\Leftrightarrow$ $\mu_{M^+}$ is an isomorphism.

We conclude that $M\in\mathcal{B}_C(R)\Leftrightarrow M^+\in\mathcal{A}_{C}(R^{op})$.
\end{proof}

As a consequence, we get the following

\begin{theorem}\label{3.3}
\begin{enumerate}
\item[]
\item[(1)] The pair
$$(\mathcal{A}_C(R^{op}),\mathcal{B}_C(R))$$ is a perfect coproduct-closed and product-closed duality pair
and $\mathcal{A}_C(R^{op})$ is covering and preenveloping in $\Mod R^{op}$.
\item[(2)] The pair
$$(\mathcal{B}_C(R),\mathcal{A}_C(R^{op}))$$ is a coproduct-closed and product-closed duality pair
and $\mathcal{B}_C(R)$ is covering and preenveloping in $\Mod R$.
\end{enumerate}
\end{theorem}

\begin{proof}
It follows from \cite[Proposition 4.2(a)]{HW07} that both $\mathcal{A}_C(R^{op})$ and $\mathcal{B}_C(R)$ are closed under
direct summands, coproducts and products. So by Lemma \ref{2.4}(1)(2) and Proposition \ref{3.2}, we have that
both the pairs
$$(\mathcal{A}_C(R^{op}),\mathcal{B}_C(R))\ {\rm and}\ (\mathcal{B}_C(R),\mathcal{A}_C(R^{op}))$$
are coproduct-closed and product-closed duality pairs, $\mathcal{A}_C(R^{op})$ is covering and preenveloping in $\Mod R^{op}$
and $\mathcal{B}_C(R)$ is covering and preenveloping in $\Mod R$. Moreover, $\mathcal{A}_C(R^{op})$ is projectively resolving
by \cite[Theorem 6.2]{HW07}, so the duality pair $(\mathcal{A}_C(R^{op}),\mathcal{B}_C(R))$ is perfect.
\end{proof}

We write
$$\mathcal{A}_C(R^{op})^{\bot}:=\{Y\in\Mod R^{op}\mid \Ext_{R^{op}}^{\geq 1}(N,Y)=0\ \text{for any}\ N\in\mathcal{A}_C(R^{op})\}.$$
The following corollary was proved in \cite[Theorem 3.11]{EH} when $R$ is a commutative noetherian ring and $_RC_S={_RC_R}$.

\begin{corollary}\label{3.4}
The pair $$(\mathcal{A}_C(R^{op}),\mathcal{A}_C(R^{op})^{\bot})$$ is a hereditary perfect cotorsion pair
and $\mathcal{A}_C(R^{op})$ is covering and preenveloping in $\Mod R^{op}$.
\end{corollary}

\begin{proof}
It follows from Theorem \ref{3.3}(1) and Lemma \ref{2.4}(3).
\end{proof}

The following two results are the symmetric versions of Theorem \ref{3.3} and Corollary \ref{3.4} respectively.

\begin{theorem}\label{3.5}
\begin{enumerate}
\item[]
\item[(1)] The pair $$(\mathcal{A}_C(S),\mathcal{B}_C(S^{op}))$$ is a perfect coproduct-closed and product-closed duality pair
and $\mathcal{A}_C(S)$ is covering and preenveloping in $\Mod S$.
\item[(2)] The pair $$(\mathcal{B}_C(S^{op}),\mathcal{A}_C(S))$$ is a coproduct-closed and product-closed duality pair
and $\mathcal{B}_C(S^{op})$ is covering and preenveloping in $\Mod S^{op}$.
\end{enumerate}
\end{theorem}

We write
$$\mathcal{A}_C(S)^{\bot}:=\{X\in\Mod S\mid \Ext_{S}^{\geq 1}(N^{'},X)=0\ \text{for any}\ N^{'}\in\mathcal{A}_C(S)\}.$$

\begin{corollary}\label{3.6}
The pair $$(\mathcal{A}_C(S),\mathcal{A}_C(S)^{\bot})$$ is a hereditary perfect cotorsion pair
and $\mathcal{A}_C(S)$ is covering and preenveloping in $\Mod S$.
\end{corollary}

Holm and White proved in \cite[Proposition 4.1]{HW07} that there exist the following (Foxby) equivalences of categories
$$\xymatrix@C=16ex{{\mathcal{A}_C(S)}\ar@<0.8ex>[r]_-{\sim}^-{C\otimes_{S}-}&
{\mathcal{B}_{C}(R)}\ar@<0.8ex>[l]^-{\Hom_{R}(C,-)},&}$$
$$\xymatrix@C=16ex{{\mathcal{A}_C(R^{op})}\ar@<0.8ex>[r]_-{\sim}^-{-\otimes_{R}C}&
{\mathcal{B}_{C}(S^{op})}\ar@<0.8ex>[l]^-{\Hom_{S^{op}}(C,-)}.&}$$
Compare this result with Theorems \ref{3.3} and \ref{3.5}.

By Theorems \ref{3.3}(2) and \ref{3.5}(1), $\mathcal{B}_C(R)$ is preenveloping in $\Mod R$
and $\mathcal{A}_C(S)$ is preenveloping in $\Mod S$. In the following result, we construct
an $\mathcal{A}_C(S)$-preenvelope of a given module in $\Mod S$ from a $\mathcal{B}_C(R)$-preenvelope
of some module in $\Mod R$.

\begin{theorem}\label{3.7}
\begin{enumerate}
\item[]
\item[(1)] Let $N\in\Mod S$ and
$$f:C\otimes_SN\to B$$ be a $\mathcal{B}_C(R)$-preenvelope of $C\otimes_SN$ in $\Mod R$.
Then we have
\begin{enumerate}
\item[(1.1)] $$f_*\mu_N:N\to B_*$$ is an $\mathcal{A}_C(S)$-preenvelope of $N$ in $\Mod S$.
\item[(1.2)] If $f$ is a $\mathcal{B}_C(R)$-envelope of $C\otimes_SN$, then $f_*\mu_N$ is an $\mathcal{A}_C(S)$-envelope of $N$.
\end{enumerate}
\item[(2)] If $\mathcal{B}_C(R)$ is enveloping in $\Mod R$, then $\mathcal{A}_C(S)$ is enveloping in $\Mod S$.
\end{enumerate}
\end{theorem}

\begin{proof}
(1.1) Let $N\in\Mod S$ and
$$f:C\otimes_SN\to B$$ be a $\mathcal{B}_C(R)$-preenvelope in $\Mod R$.
By \cite[Proposition 4.1]{HW07}, we have $B_*\in\mathcal{A}_C(S)$.
Let $g\in\Hom_S(N,A)$ with $A\in\mathcal{A}_C(S)$. By \cite[Proposition 4.1]{HW07}
again, we have $C\otimes_SA\in\mathcal{B}_C(R)$. So there exists $h\in\Hom_R(B,C\otimes_SA)$
such that $1_C\otimes g=hf$, that is, the following diagram
$$\xymatrix@R=15pt@C=15pt{
&&C\otimes_SN\ar[d]_{1_C\otimes g}\ar[r]^{\ \ \ \ f}&B\ar@{-->}[ld]^{h}&\\
&&C\otimes_SA&&}$$
commutes. From the following commutative diagram
$$\xymatrix@R=20pt@C=20pt{
N\ar[r]^{g}\ar[d]_{\mu_N} &A\ar[d]^{\mu_A}\\
(C\otimes_SN)_*\ar[r]^{(1_C\otimes g)_*}&(C\otimes_SA)_*,}$$
we get $\mu_Ag=(1_C\otimes g)_*\mu_N$. Because $\mu_A$ is an isomorphism, we have
$$g={\mu_A}^{-1}(1_C\otimes g)_*\mu_N=({\mu_A}^{-1}h_*)(f_*\mu_N),$$
that is, the following diagram
$$\xymatrix@R=15pt@C=15pt{
&&N\ar[d]_{g}\ar[r]^{f_*\mu_N}&B\ar@{-->}[ld]^{{\mu_A}^{-1}h_*}&\\
&&A&&}$$
commutes. Thus $f_*\mu_N:N\to B_*$ is an $\mathcal{A}_C(S)$-preenvelope of $N$.

(1.2) By (1.1), it suffices to prove that if $f$ is left minimal, then so is $f_*\mu_N$.

Let $f$ be left minimal and $h\in\Hom_S(B_*,B_*)$ such that $f_*\mu_N=h(f_*\mu_N)$. Then we have
$$(1_C\otimes f_*)(1_C\otimes \mu_N)=1_C\otimes (f_*\mu_N)=1_C\otimes (h(f_*\mu_N))
=(1_C\otimes h)(1_C\otimes f_*)(1_C\otimes \mu_N).\eqno{(3.1)}$$
From the following commutative diagram
$$\xymatrix@R=20pt@C=20pt{
C\otimes_S(C\otimes_SN)_*\ar[r]^{1_C\otimes f_*}\ar[d]_{\theta_{C\otimes_SN}} &C\otimes_SB_*\ar[d]^{\theta_B}\\
C\otimes_SN\ar[r]^{f}&B,}$$
we get
$$f\theta_{C\otimes_SN}=\theta_B(1_C\otimes f_*).\eqno{(3.2)}$$
So we have
\begin{align*}
&\ \ \ \ \ \ \ \ \ \ f=f1_{C\otimes_SN}\\
&\ \ \ \ \ \ \ \ \ \ \ \ =f(\theta_{C\otimes_SN}(1_C\otimes \mu_N))\ \text{(by \cite[Proposition 2.2(1)]{Wi}})\\
&\ \ \ \ \ \ \ \ \ \ \ \ =\theta_B(1_C\otimes f_*)(1_C\otimes \mu_N)\ \text{(by (3.2))}\\
&\ \ \ \ \ \ \ \ \ \ \ \ =\theta_B(1_C\otimes h)(1_C\otimes f_*)(1_C\otimes \mu_N)\ \text{(by (3.1))}\\
&\ \ \ \ \ \ \ \ \ \ \ \ =\theta_B(1_C\otimes h)({\theta_B}^{-1}\theta_B)(1_C\otimes f_*)(1_C\otimes \mu_N)\ \text{(because $\theta_B$ is an isomorphism)}\\
&\ \ \ \ \ \ \ \ \ \ \ \ =\theta_B(1_C\otimes h){\theta_B}^{-1}f\theta_{C\otimes_SN}(1_C\otimes \mu_N)\ \text{(by (3.2))}\\
&\ \ \ \ \ \ \ \ \ \ \ \ =\theta_B(1_C\otimes h){\theta_B}^{-1}f1_{C\otimes_SN}\ \text{(by \cite[Proposition 2.2(1)]{Wi}})\\
&\ \ \ \ \ \ \ \ \ \ \ \ =\theta_B(1_C\otimes h){\theta_B}^{-1}f.
\end{align*}
Because $f$ is left minimal, $\theta_B(1_C\otimes h){\theta_B}^{-1}$ is an isomorphism, which implies that
$1_C\otimes h$ and $(1_C\otimes h)_*$ are also isomorphisms. From the following commutative diagram
$$\xymatrix@R=20pt@C=20pt{
B_*\ar[r]^{h}\ar[d]_{\mu_{B_*}} &B_*\ar[d]^{\mu_{B_*}}\\
(C\otimes_SB_*)_*\ar[r]^{(1_C\otimes h)_*}&(C\otimes_SB_*)_*,}$$
we get
$$(1_C\otimes h)_*\mu_{B_*}=\mu_{B_*}h.$$
Because $B_*\in\mathcal{A}_C(S)$ by \cite[Proposition 4.1]{HW07}, $\mu_{B_*}$ is an isomorphism.
It follows that $h$ is also an isomorphism and $f_*\mu_N$ is left minimal.

(2) It follows from the assertion (1.2) immediately.
\end{proof}

We do not know whether a $\mathcal{B}_C(R)$-preenvelope of given module in $\Mod R$
can be constructed from an $\mathcal{A}_C(S)$-preenvelope of some module in $\Mod S$,
and do not know whether the converse of Theorem \ref{3.7}(2) holds true.

By Theorems \ref{3.3}(2) and \ref{3.5}(1), $\mathcal{B}_C(R)$ is covering in $\Mod R$
and $\mathcal{A}_C(S)$ is covering in $\Mod S$. In the following result, we construct
a $\mathcal{B}_C(R)$-cover of a given module in $\Mod R$ from an $\mathcal{A}_C(S)$-cover
of some module in $\Mod S$.

\begin{proposition}\label{3.8}
Let $M\in\Mod R$ and $$g:A\to M_*$$ be an $\mathcal{A}_C(S)$-cover of $M_*$ in $\Mod S$.
Then
$$\theta_M(1_C\otimes g):C\otimes_SA\to M$$ is a $\mathcal{B}_C(R)$-cover of $M$ in $\Mod R$.
\end{proposition}

\begin{proof}
Let $M\in\Mod R$ and $$g:A\to M_*$$ be an $\mathcal{A}_C(S)$-cover of $M_*$ in $\Mod S$.
By \cite[Proposition 4.1]{HW07}, we have $C\otimes_SA\in\mathcal{B}_C(R)$.
Let $f\in\Hom_R(B,M)$ with $B\in\mathcal{B}_C(R)$. By \cite[Proposition 4.1]{HW07}
again, we have $B_*\in\mathcal{A}_C(S)$. So there exists $h\in\Hom_S(B_*,A)$
such that ${f}_*=gh$, that is, the following diagram
$$\xymatrix{ & B_* \ar[d]^{{f}_*} \ar@{-->}[ld]_{h}\\
A \ar[r]^{g} & M_*}$$
commutes. From the following commutative diagram
$$\xymatrix@R=20pt@C=20pt{
C\otimes_SB_*\ar[r]^{1_C\otimes {f}_*}\ar[d]_{\theta_B} &C\otimes_SM_*\ar[d]^{\theta_M}\\
B\ar[r]^{f}&M,}$$
we get $f\theta_B=\theta_M(1_C\otimes {f}_*)$. Because $\theta_B$ is an isomorphism, we have
$$f=\theta_M(1_C\otimes {f}_*){\theta_B}^{-1}=\theta_M(1_C\otimes (gh)){\theta_B}^{-1}
=(\theta_M(1_C\otimes g))((1_C\otimes h)){\theta_B}^{-1}),$$
that is, the following diagram
$$\xymatrix{ & B \ar[d]^{f} \ar@{-->}[ld]_{(1_C\otimes h)){\theta_B}^{-1}}\\
C\otimes_SA \ar[r]_{\theta_M(1_C\otimes g)} & M}$$
commutes. Thus $\theta_M(1_C\otimes g):C\otimes_SA\to M$ is a $\mathcal{B}_C(R)$-precover of $M$.

In the following, it suffices to prove that $\theta_M(1_C\otimes g)$ is right minimal.

Let $h\in\Hom_R(C\otimes_SA,C\otimes_SA)$ such that $\theta_M(1_C\otimes g)=(\theta_M(1_C\otimes g))h$. Then we have
$$(\theta_M)_*(1_C\otimes g)_*=(\theta_M(1_C\otimes g))_*=((\theta_M(1_C\otimes g))h)_*
=(\theta_M)_*(1_C\otimes g)_*h_*.\eqno{(3.3)}$$
From the following commutative diagram
$$\xymatrix@R=20pt@C=20pt{
A\ar[r]^{g}\ar[d]_{\mu_A} &M_*\ar[d]^{\mu_{M_*}}\\
(C\otimes_SA)_*\ar[r]^{(1_C\otimes g)_*}&(C\otimes_SM_*)_*,}$$
we get
$$\mu_{M_*}g=(1_C\otimes g)_*\mu_A.\eqno{(3.4)}$$
So we have
\begin{align*}
&\ \ \ \ \ \ \ \ \ \ g=1_{M_*}g\\
&\ \ \ \ \ \ \ \ \ \ \ \ =(\theta_M)_*\mu_{M_*}g\ \text{(by \cite[Proposition 2.2(1)]{Wi}})\\
&\ \ \ \ \ \ \ \ \ \ \ \ =(\theta_M)_*(1_C\otimes g)_*\mu_A\ \text{(by (3.4))}\\
&\ \ \ \ \ \ \ \ \ \ \ \ =(\theta_M)_*(1_C\otimes g)_*h_*\mu_A\ \text{(by (3.3))}\\
&\ \ \ \ \ \ \ \ \ \ \ \ =(\theta_M)_*(1_C\otimes g)_*\mu_A{\mu_A}^{-1}h_*\mu_A\ \text{(because $\mu_A$ is an isomorphism)}\\
&\ \ \ \ \ \ \ \ \ \ \ \ =(\theta_M)_*\mu_{M_*}g{\mu_A}^{-1}h_*\mu_A\ \text{(by (3.4))}\\
&\ \ \ \ \ \ \ \ \ \ \ \ =1_{M_*}g{\mu_A}^{-1}h_*\mu_A\ \text{(by \cite[Proposition 2.2(1)]{Wi}})\\
&\ \ \ \ \ \ \ \ \ \ \ \ =g{\mu_A}^{-1}h_*\mu_A.
\end{align*}
Because $g$ is right minimal, ${\mu_A}^{-1}h_*\mu_A$ is an isomorphism, which implies that
$h_*$ and $1_C\otimes h_*$ are also isomorphisms. From the following commutative diagram
$$\xymatrix@R=20pt@C=20pt{
C\otimes_S(C\otimes_SA)_*\ar[r]^{1_C\otimes h_*}\ar[d]_{\theta_{C\otimes_SA}} &C\otimes_S(C\otimes_SA)_*\ar[d]^{\theta_{C\otimes_SA}}\\
C\otimes_SA\ar[r]^{h}&C\otimes_SA,}$$
we get
$$h\theta_{C\otimes_SA}=\theta_{C\otimes_SA}(1_C\otimes h_*).$$
Because $C\otimes_SA\in\mathcal{B}_C(R)$ by \cite[Proposition 4.1]{HW07}, $\theta_{C\otimes_SA}$ is an isomorphism.
It follows that $h$ is also an isomorphism and $\theta_M(1_C\otimes g)$ is right minimal.
\end{proof}

We do not know whether an $\mathcal{A}_C(S)$-cover of a given module in $\Mod S$
can be constructed from a $\mathcal{B}_C(R)$-cover of some module in $\Mod R$.

\section{The Auslander projective dimension of modules}

For a subcategory $\mathscr{X}$ of $\Mod S$ and $N\in \Mod S$, the \textbf{$\mathscr{X}$-projective dimension}
$\mathscr{X}$-$\pd_SN$ of $N$ is defined as $\inf\{n\mid$ there exists an exact sequence
$$0 \to X_n \to \cdots \to X_1\to X_0 \to N\to 0$$
in $\Mod S$ with all $X_i\in\mathscr{X}\}$, and we set $\mathscr{X}$-$\pd_SN$ infinite
if no such integer exists. We call $\mathcal{A}_{C}(S)$-$\pd_SN$ the \textbf{Auslander projective dimension}
of $N$. For any $n\geq 0$, we use $\Omega^n(N)$ to denote the $n$-th syzygy of $N$ (note: $\Omega^0(N)=N$).

\begin{lemma}\label{4.1}
Let $N\in\Mod S$ and $n\geq 0$. If $\mathcal{A}_{C}(S)$-$\pd_SN\leq n$ and
$$0\to K_n \to A_{n-1} \to \cdots \to A_1\to A_0 \to N\to 0$$
be an exact sequence in $\Mod S$ with all $A_i$ in $\mathcal{A}_{C}(S)$,
then $K_n\in\mathcal{A}_{C}(S)$; in particular, $\Omega^n(N)\in\mathcal{A}_C(S)$.
\end{lemma}

\begin{proof}
Because $\mathcal{A}_{C}(S)$ is projectively resolving and is closed under direct summands and coproducts
by \cite[Theorem 6.2 and Proposition 4.2(a)]{HW07}, the assertion follows from \cite[Lemma 3.12]{AB}.
\end{proof}

We use $\mathcal{A}_{C}(S)$-$\pd^{<\infty}$ to denote the subcategory of $\Mod S$ consisting of modules
with finite Auslander projective dimension.

\begin{proposition}\label{4.2}
$\mathcal{A}_{C}(S)$-$\pd^{<\infty}$ is closed under extensions, kernels of epimorphisms and
cokernels of monomorphisms.
\end{proposition}

\begin{proof}
Let $$0\to N_1 \to N_2 \to N_3 \to 0$$
be an exact sequence in $\Mod S$ and $n\geq 0$. If
$\max\{\mathcal{A}_{C}(S)$-$\pd_SN_1,\mathcal{A}_{C}(S)$-$\pd_SN_3\}\leq n$,
then by Lemma \ref{4.1}, there exist exact sequences
$$0\to \Omega^n(N_1)\to P^{n-1}_1 \to \cdots\to P^{1}_1 \to P^{0}_1 \to N_1 \to 0,$$
$$0\to \Omega^n(N_3)\to P^{n-1}_3 \to \cdots\to P^{1}_3 \to P^{0}_3 \to N_3 \to 0$$
in $\Mod S$ with all $P_i^j$ projective and $\Omega^n(N_1),\Omega^n(N_3)\in\mathcal{A}_{C}(S)$.
Then we get exact sequences
$$0\to K_n\to P^{n-1}_1\oplus P^{n-1}_3 \to \cdots\to P^{1}_1 \oplus P^{1}_3\to P^{0}_1\oplus P^{0}_3 \to N_2 \to 0,$$
$$0\to \Omega^n(N_1)\to K_n \to  \Omega^n(N_3)\to 0$$
in $\Mod S$. By \cite[Theorem 6.2]{HW07}, we have $K_n\in\mathcal{A}_{C}(S)$ and
$\mathcal{A}_{C}(S)$-$\pd_SN_2\leq n$.

If $\max\{\mathcal{A}_{C}(S)$-$\pd_SN_1,\mathcal{A}_{C}(S)$-$\pd_SN_2\}\leq n$, then by Corollary \ref{3.6}
and Lemma \ref{4.1}, there exist $\Hom_S(\mathcal{A}_{C}(S),-)$-exact exact sequences
$$0\to A^{n}_1\to A^{n-1}_1 \to \cdots\to A^{1}_1 \to A^{0}_1 \to N_1 \to 0,$$
$$0\to A^{n}_2\to A^{n-1}_2 \to \cdots\to A^{1}_2 \to A^{0}_2 \to N_2 \to 0$$
in $\Mod S$ with all $A_i^j$ in $\mathcal{A}_{C}(S)$. By \cite[Theorem 3.6]{Hu}, we get an exact sequence
$$0\to A^n_1\to A^{n-1}_1\oplus A^{n}_2 \to \cdots\to A^{0}_1 \oplus A^{1}_2\to A^{0}_2 \to N_3 \to 0$$
in $\Mod S$, and so $\mathcal{A}_{C}(S)$-$\pd_SN_3\leq n+1$.

If $\max\{\mathcal{A}_{C}(S)$-$\pd_SN_2,\mathcal{A}_{C}(S)$-$\pd_SN_3\}\leq n$, then by Corollary \ref{3.6}
and Lemma \ref{4.1}, there exist $\Hom_S(\mathcal{A}_{C}(S),-)$-exact exact sequences
$$0\to A^{n}_2\to A^{n-1}_2 \to \cdots\to A^{1}_2 \to A^{0}_2 \to N_2 \to 0,$$
$$0\to A^{n}_3\to A^{n-1}_3 \to \cdots\to A^{1}_3 \to A^{0}_3 \to N_3 \to 0$$
in $\Mod S$ with all $A_i^j$ in $\mathcal{A}_{C}(S)$. By \cite[Theorem 3.2]{Hu}, we get exact sequences
$$0\to A^n_2\to A^{n-1}_2\oplus A^{n}_3 \to \cdots\to A^{1}_2 \oplus A^{2}_3\to A \to N_1 \to 0,$$
$$0\to A \to A_{2}^0 \oplus A_{3}^1\to A^{0}_3 \to 0$$
in $\Mod S$. By \cite[Theorem 6.2]{HW07}, we have $A\in\mathcal{A}_{C}(S)$,
and so $\mathcal{A}_{C}(S)$-$\pd_SN_1\leq n$.
\end{proof}

We write
$$\mathcal{I}_C(S):=\{I_*\mid I\ {\rm \ is\ injective\ in}\ \Mod R\}.$$
The modules in $\mathcal{I}_C(S)$ is called {\bf $C$-injective} (\cite{HW07}).
Let $Q$ be an injective cogenerator for $\Mod R$. Then
$$\mathcal{I}_C(S)=\Prod_SQ_*$$ by \cite[Proposition 2.4(2)]{LHX},
where $\Prod_SQ_*$ is the subcategory of $\Mod S$ consisting of direct summands
of products of copies of $Q_*$. By \cite[Lemma 2.16(b)]{GT12}, we have the following
isomorphism of functors
$$\Hom_R(\Tor_i^S(C,-),Q)\cong\Ext_S^i(-,Q_*)$$
for any $i\geq 1$. This gives the following

\begin{lemma}\label{4.3}
${C_S}^{\top}={^{\bot}\mathcal{I}_C(S)}$.
\end{lemma}

For a subcategory $\mathscr{X}$ of $\Mod S$, a sequence in $\Mod S$ is called
{\bf $\Hom_{S}(-,\mathscr{X})$-exact} if it is exact after applying the functor
$\Hom_{S}(-,X)$ for any $X\in\mathscr{X}$.
Now we give some criteria for computing the Auslander projective dimension of modules.

\begin{theorem}\label{4.4}
Let $N\in\Mod S$ with $\mathcal{A}_{C}(S)$-$\pd_SN<\infty$ and $n\geq 0$.
Then the following statements are equivalent.
\begin{enumerate}
\item[(1)] $\mathcal{A}_{C}(S)$-$\pd_SN\leq n$.
\item[(2)] $\Omega^n(N)\in\mathcal{A}_{C}(S)$.
\item[(3)] $\Tor^S_{\geq n+1}(C,N)=0$.
\item[(4)] There exists an exact sequence
$$0\to H \to A \to N\to 0$$
in $\Mod S$ with $A\in \mathcal{A}_{C}(S)$ and $\mathcal{I}_C(S)$-$\pd_SH\leq n-1$.
\item[(5)] There exists a ($\Hom_{S}(-,\mathcal{I}_C(S))$-exact) exact sequence
$$0\to N \to H^{'} \to A^{'} \to 0$$
in $\Mod S$ with $A^{'}\in \mathcal{A}_{C}(S)$ and $\mathcal{I}_C(S)$-$\pd_SH^{'}\leq n$.
\end{enumerate}
\end{theorem}

\begin{proof}
By Lemma \ref{4.1} and the dimension shifting, we have $(1)\Leftrightarrow (2)\Rightarrow (3)$.

$(3)\Rightarrow (2)$ Because $\Tor^S_{\geq n+1}(C,N)=0$ by (3), we have $\Omega^n(N)\in{C_S}^\top$,
and so $\Omega^n(N)\in{^{\bot}\mathcal{I}_C(S)}$ by Lemma \ref{4.3}.
Note that all projective modules in $\Mod S$ are in $\mathcal{A}_C(S)$ by \cite[Theorem 6.2]{HW07}.
Because $\mathcal{A}_{C}(S)$-$\pd_SN<\infty$ by assumption,
we have $\mathcal{A}_{C}(S)$-$\pd_S\Omega^n(N)<\infty$ by Proposition 4.2.

Assume that $\mathcal{A}_{C}(S)$-$\pd_S\Omega^n(N)=m(<\infty)$ and
$$0\to A_m \to \cdots \to A_1 \to A_0 \to \Omega^n(N) \to 0 \eqno{(4.1)}$$
is an exact sequence in $\Mod S$ with all $A_j$ in $\mathcal{A}_C(S)$.
Because $\mathcal{A}_C(S)\subseteq {C_S}^{\top}={^{\bot}\mathcal{I}_C(S)}$ by
Lemma \ref{4.3}, the exact sequence (4.1) is $\Hom_{S}(-,\mathcal{I}_C(S))$-exact.
By \cite[Theorem 3.11(1)]{TH}, we have the following $\Hom_{S}(-,\mathcal{I}_C(S))$-exact exact sequence
$$0\to A_j \to U_j^0\to U_j^1 \to \cdots \to U_j^i \to \cdots$$
in $\Mod S$ with all $U_j^i$ in $\mathcal{I}_C(S)$ for any $0\leq j\leq m$ and $i\geq 0$.
It follows from \cite[Corollary 3.5]{Hu} that there exist
the following two exact sequences
$$0\to \Omega^n(N) \to U \to \oplus _{i=0}^mU_i^{i+1} \to \oplus_{i=0}^mU_i^{i+2}
\to \oplus_{i=0}^mU_i^{i+3} \to \cdots,$$
$$0\to U_m^0 \to U_m^1\oplus U_{m-1}^0 \to \cdots \to \oplus _{i=2}^mU_i^{i-2}
\to \oplus_{i=1}^mU_i^{i-1} \to \oplus _{i=0}^mU_i^i \to U \to 0,$$
and the former one is $\Hom_{S}(-,\mathcal{I}_C(S))$-exact. Because $\mathcal{I}_C(S)$ is closed
under finite direct sums and cokernels of monomorphisms by \cite[Proposition 5.1(c) and Corollary 6.4]{HW07},
we have $U\in\mathcal{I}_{C}(S)$. By \cite[Theorem 3.11(1)]{TH} again, we have
$\Omega^n(N)\in\mathcal{A}_{C}(S)$.

$(1)\Rightarrow (4)$ By \cite[Theorem 6.2]{HW07}, $\mathcal{A}_C(S)$ is closed under extensions.
By \cite[Theorem 3.11(1)]{TH}, we have that
$\mathcal{I}_C(S)$ is an $\mathcal{I}_C(S)$-coproper cogenerator
for $\mathcal{A}_C(S)$ in the sense of \cite{Hu2}. Then the assertion follows from \cite[Theorem 4.7]{Hu2}.

$(4)\Rightarrow (5)$ Let
$$0\to H \to A \to N\to 0$$
be an exact sequence in $\Mod S$ with $A\in \mathcal{A}_{C}(S)$ and $\mathcal{I}_C(S)$-$\pd_SH\leq n-1$.
By \cite[Theorem 3.11(1)]{TH}, there exists a $\Hom_{S}(-,\mathcal{I}_C(S))$-exact exact sequence
$$0\to A\to U\to A^{'}\to 0$$ in $\Mod S$ with $U\in\mathcal{I}_C(S)$
and $A^{'}\in \mathcal{A}_{C}(S)$. Consider the following push-out diagram
$$\xymatrix{
& & 0 \ar[d] &0 \ar@{-->}[d] &   \\
0 \ar[r]  & H \ar@{==}[d]  \ar[r] & A \ar[d]\ar[r]  &N \ar@{-->}[d] \ar[r] & 0 \\
0 \ar@{-->}[r] & H \ar@{-->}[r] & U \ar[d] \ar@{-->}[r] & H^{'} \ar@{-->}[d] \ar@{-->}[r] & 0  \\
&  & A^{'} \ar[d]\ar@{==}[r] & A^{'}\ar@{-->}[d]  &   \\
&  & 0  & 0.   &   } $$
By the middle row in this diagram, we have $\mathcal{I}_C(S)$-$\pd_SH^{'}\leq n$. Because the middle column
in the above diagram is $\Hom_{S}(-,\mathcal{I}_C(S))$-exact, the rightmost column is also
$\Hom_{S}(-,\mathcal{I}_C(S))$-exact by \cite[Lemma 2.4(2)]{Hu} and it is the desired exact sequence.

$(5)\Rightarrow (1)$ Let
$$0\to N \to H^{'} \to A^{'} \to 0$$
be an exact sequence in $\Mod S$ with $A^{'}\in \mathcal{A}_{C}(S)$ and $\mathcal{I}_C(S)$-$\pd_SH^{'}\leq n$.
Then there exists an exact sequence
$$0\ra U_n\ra \cdots\ra U_1\ra U_0\ra H^{'}\ra 0$$ in $\Mod S$ with all $U_i$ in $\mathcal{I}_C(S)$.
Set $H:=\mbox{Ker}(U_0\ra H^{'})$.
Then $\mathcal{I}_C(S)$-$\pd_SH\leq n-1$. Consider the following pull-back diagram
$$\xymatrix@R=20pt@C=20pt{& 0 \ar@{-->}[d] & 0 \ar[d]&& &\\
& H \ar@{==}[r] \ar@{-->}[d] & H \ar[d]& &&\\
0 \ar@{-->}[r] & A \ar@{-->}[d] \ar@{-->}[r] & U_0 \ar[d] \ar@{-->}[r] &A^{'} \ar@{==}[d] \ar@{-->}[r] & 0\\
0 \ar[r] & N \ar@{-->}[d]\ar[r] & H^{'} \ar[r] \ar[d] & A^{'} \ar[r] & 0 &\\
& 0 & 0. & && }$$
Applying \cite[Theorem 6.2]{HW07} to the middle row in this diagram yields $A\in \mathcal{A}_{C}(S)$.
Thus $\mathcal{A}_{C}(S)$-$\pd_SN\leq n$ by the leftmost column in the above diagram.
\end{proof}

The only place where the assumption $\mathcal{A}_{C}(S)$-$\pd_SN<\infty$ in Theorem \ref{4.4} is used is in showing
$(3)\Rightarrow(2)$. By Theorem \ref{4.4}, it is easy to get the following standard observation.

\begin{corollary}\label{4.5}
Let $$0\to L \to M \to K\to 0$$
be an exact sequence in $\Mod S$. Then we have
\begin{enumerate}
\item[(1)] $\mathcal{A}_{C}(S)$-$\pd_{S}K \leq \max\{\mathcal{A}_{C}(S)$-$\pd_{S}M,\mathcal{A}_{C}(S)$-$\pd_{S}L+1\}$,
and the equality holds true if $\mathcal{A}_{C}(S)$-$\pd_{S}M \neq \mathcal{A}_{C}(S)$-$\pd_{S}L$.
\item[(2)] $\mathcal{A}_{C}(S)$-$\pd_{S}L \leq \max\{\mathcal{A}_{C}(S)$-$\pd_{S}M,\mathcal{A}_{C}(S)$-$\pd_{S}K-1\}$,
and the equality holds true if $\mathcal{A}_{C}(S)$-$\pd_{S}M\neq \mathcal{A}_{C}(S)$-$\pd_{S}K$.
\item[(3)] $\mathcal{A}_{C}(S)$-$\pd_{S}M \leq \max\{\mathcal{A}_{C}(S)$-$\pd_{S}L,\mathcal{A}_{C}(S)$-$\pd_{S}K\}$,
and the equality holds true if $\mathcal{A}_{C}(S)$-$\pd_{S}K \neq \mathcal{A}_{C}(S)$-$\pd_{S}L+1$.
\end{enumerate}
\end{corollary}

The following corollary is an addendum to the implications $(1)\Rightarrow (4)$ and $(1)\Rightarrow (5)$ in Theorem \ref{4.4}.

\begin{corollary}\label{4.6}
Let $N\in\Mod S$ with $\mathcal{A}_{C}(S)$-$\pd_SN=n(<\infty)$.
Then there exist exact sequences
$$0\to H \to A \to N\to 0,$$
$$0\to N \to H^{'} \to A^{'} \to 0$$
in $\Mod S$ with $A,A^{'}\in \mathcal{A}_{C}(S)$ and $\mathcal{I}_C(S)$-$\pd_SH=\mathcal{I}_C(S)$-$\pd_SH^{'}=n$.
\end{corollary}

\begin{proof}
Let $N\in\Mod S$ with $\mathcal{A}_{C}(S)$-$\pd_SN=n(<\infty)$.
By Theorem \ref{4.4}, there exists an exact sequence
$$0\to H \to A \to N\to 0$$
in $\Mod S$ with $A\in \mathcal{A}_{C}(S)$ and $(\mathcal{A}_{C}(S)$-$\pd_SH\leq)\mathcal{I}_C(S)$-$\pd_SH\leq n-1$.
By Theorem \ref{4.4} again, we have $\sup\{i\geq 0\mid \Tor^S_{i}(C,N)\neq 0\}=n$.
So $\sup\{i\geq 0\mid \Tor^S_{i}(C,H)\neq 0\}=n-1$, and hence
$\mathcal{A}_{C}(S)$-$\pd_SH=n-1$ by Theorem \ref{4.4}.
It follows that $\mathcal{I}_C(S)$-$\pd_SH=n-1$.

By Theorem \ref{4.4}, there exists an exact sequence
$$0\to N \to H^{'} \to A^{'} \to 0$$
in $\Mod S$ with $A^{'}\in \mathcal{A}_{C}(S)$ and $(\mathcal{A}_{C}(S)$-$\pd_SH\leq)\mathcal{I}_C(S)$-$\pd_SH^{'}\leq n$.
By Corollary \ref{4.5}(3), we have $\mathcal{A}_{C}(S)$-$\pd_SH=\mathcal{A}_{C}(S)$-$\pd_SN=n$,
and so $\mathcal{I}_C(S)$-$\pd_SH^{'}=n$.
\end{proof}

Let $N\in \Mod S$. Bican, El Bashir and Enochs proved in \cite{BBE} that $N$ has a flat cover. We use
$$\cdots \buildrel {f_{n+1}} \over \longrightarrow F_n(N) \buildrel {f_n} \over \longrightarrow \cdots
\buildrel {f_2} \over \longrightarrow F_1(N) \buildrel {f_1} \over \longrightarrow F_0(N) \buildrel {f_0}
\over \longrightarrow N \to 0\eqno{(4.2)}$$
to denote a minimal flat resolution of $N$ in $\Mod S$, where each $F_i(N)\to\Im f_i$ is a flat cover of $\Im f_i$.

\begin{lemma}\label{4.7}
Let $N\in\Mod S$ and $n\geq 0$. If $\Tor_{1\leq i\leq n}^S(C,N)=0$, then we have
\begin{enumerate}
\item[(1)] There exists an exact sequence
$$0\rightarrow \Ext_R^{n+1}(C,\Ker(1_C\otimes f_{n+1}))\rightarrow N\stackrel{\mu_{N}}{\longrightarrow}
(C\otimes_SN)_*\rightarrow\Ext_R^{n+2}(C,\Ker(1_C\otimes f_{n+1}))\rightarrow 0$$
in $\Mod S$.
\item[(2)] $\Ext_R^{1\leq i \leq n}(C,\Ker(1_C\otimes f_{n+1}))=0$.
\end{enumerate}
\end{lemma}

\begin{proof}
(1) The case for $n=0$ follows from \cite[Proposition 3.2]{TH}. Now suppose $n\geq 1$. If $\Tor_{1\leq i\leq n}^S(C,N)=0$,
then the exact sequence (4.2) yields the following exact sequence
$$0\to \Ker(1_C\otimes f_{n+1}) \to C\otimes_SF_{n+1}(N) \buildrel {1_C\otimes f_{n+1}}\over \longrightarrow
C\otimes_SF_{n}(N) \buildrel {1_C\otimes f_{n}}\over \longrightarrow\cdots$$
$$\buildrel {1_C\otimes f_{2}}\over \longrightarrow C\otimes_SF_{1}(N) \buildrel {1_C\otimes f_{1}}
\over \longrightarrow C\otimes_SF_{0}(N) \buildrel {1_C\otimes f_{0}}\over \longrightarrow C\otimes_SN \to 0\eqno{(4.3)}$$
in $\Mod R$. Because all $C\otimes_SF_i(N)$ are in $_RC^{\bot}$ by \cite[Lemma 2.3(1)]{TH}, we have
$$\Ext_R^{1}(C,\Ker(1_C\otimes f_1))\cong \Ext_R^{n+1}(C,\Ker(1_C\otimes f_n)),$$
$$\Ext_R^{2}(C,\Ker(1_C\otimes f_1))\cong \Ext_R^{n+2}(C,\Ker(1_C\otimes f_n)).$$
Now the assertion follows from \cite[Proposition 3.2]{TH}.

(2) Applying the functor $(-)_*$ to the exact sequence (4.3) we get the following commutative diagram
$$\xymatrix{
&&F_{n+1}(N) \ar[d]^{\mu_{F_{n+1}(N)}}\ar[r]^{\ \ \ \ f_{n+1}}& F_n(N)\ar[d]^{\mu_{F_{n}(N)}}\ar[r]^{f_n}
& \cdots \ar[r]^{f_1} & F_0(N)\ar[d]^{\mu_{F_{0}(N)}}\\
0\ar[r]&(\Ker(1_C\otimes f_{n+1}))_*\ar[r]& (C\otimes_SF_{n+1}(N))_*\ar[r]^{\ \ \ \ (1_C\otimes f_{n+1})_*\ \ \ \ }
& (C\otimes_SF_{n}(N))_*\ar[r]^{\ \ \ \ \ \ (1_C\otimes f_{n})_*}\ar[r]& \cdots \ar[r]^{(1_C\otimes f_{1})_* \ \ \ \ \ \ \ }& (C\otimes_SF_{0}(N))_*.}$$
All columns are isomorphisms by \cite[Lemma 4.1]{HW07}. So the bottom row in this diagram is exact.
Because all $C\otimes_SF_{i}(N)$ are in $_RC^{\bot}$, we have
$\Ext_R^{1\leq i \leq n}(C,\Ker(1_C\otimes f_{n+1}))=0$.
\end{proof}

Let $X\in \Mod R$ and let $$\cdots \buildrel {g_{n+1}} \over \longrightarrow P_n \buildrel {g_n} \over \longrightarrow
\cdots \buildrel {g_2} \over \longrightarrow P_1 \buildrel {g_1} \over \longrightarrow P_0 \buildrel {g_0} \over \longrightarrow X \to 0$$
be a projective resolution of $X$ in $\Mod R$. If there exists $n\geq 1$ such that $\Im g_n\cong \oplus_jW_j$, where each $W_j$ is
isomorphic to a direct summand of some $\Im g_{i_j}$ with $i_j<n$, then we say that $X$ {\bf has an ultimately closed projective resolution at $n$};
and we say that $X$ {\bf has an ultimately closed projective resolution} if it has an ultimately closed projective resolution at some $n$ (\cite{J}).
It is trivial that if $\pd_{R}X$ (the projective dimension of $X$) $\leq n$, then $X$ has an ultimately closed projective resolution at $n+1$.
Let $R$ be an artin algebra. If either $R$ is of finite representation type or the square of the radical of $R$ is zero,
then any finitely generated left $R$-module has an ultimately closed projective resolution (\cite[p.341]{J}).
Following \cite{Wi}, a module $N\in\Mod S$ is called {\bf $C$-adstatic} if $\mu_N$ is an isomorphism.

\begin{proposition}\label{4.8}
Let $N\in\Mod S$ and $n\geq 1$. If $\Tor_{1\leq i\leq n}^S(C,N)=0$, then $N$ is $C$-adstatic
provided that one of the following conditions is satisfied.
\begin{enumerate}
\item[(1)] $\pd_{R}C\leq n$.
\item[(2)] $_RC$ has an ultimately closed projective resolution at $n$.
\end{enumerate}
\end{proposition}

\begin{proof} (1) It follows directly from Lemma \ref{4.7}(1).

(2) Let
$$\cdots \buildrel {g_{n+1}} \over \longrightarrow P_n \buildrel {g_n} \over \longrightarrow
\cdots \buildrel {g_2} \over \longrightarrow P_1 \buildrel {g_1} \over \longrightarrow
P_0 \buildrel {g_0} \over \longrightarrow C \to 0$$ be a projective resolution of $C$ in $\Mod R$
ultimately closed at $n$. Then $\Im g_n\cong \oplus_jW_j$ such that each $W_j$ is
isomorphic to a direct summand of some $\Im g_{i_j}$ with $i_j<n$. Let $N\in\Mod S$ with
$\Tor_{1\leq i\leq n}^S(C,N)=0$. By Lemma \ref{4.7}(2), we have
$$\Ext^{1}_R(\Im g_{i_j},\Ker(1_C\otimes f_{n+1}))\cong \Ext^{i_j+1}_R(C,\Ker(1_C\otimes f_{n+1}))=0.$$
Because $W_j$ is isomorphic to a direct summand of some $\Im g_{i_j}$, we have $\Ext^{1}_R(W_j,\Ker(1_C\otimes f_{n+1}))=0$
for any $j$, which implies
\begin{align*}
&\ \ \ \ \Ext^{n+1}_R(C,\Ker(1_C\otimes f_{n+1}))\\
& \cong \Ext^{1}_R(\Im g_{n},\Ker(1_C\otimes f_{n+1}))\\
& \cong \Ext^{1}_R(\oplus_jW_j,\Ker(1_C\otimes f_{n+1}))\\
& \cong \Pi_j\Ext^{1}_R(W_j,\Ker(1_C\otimes f_{n+1}))\\
& =0.
\end{align*}
Then by Lemma \ref{4.7}(2), we conclude that $\Ext_R^{1\leq i \leq n+1}(C,\Ker(1_C\otimes f_{n+1}))=0$.
Similar to the above argument we get $\Ext_R^{n+2}(C,\Ker(1_C\otimes f_{n+1}))=0$.
It follows from Lemma \ref{4.7}(1) that $\mu_N$ is an isomorphism and $N$ is $C$-adstatic.
\end{proof}

\begin{corollary}\label{4.9}
For any $n\geq 1$,
a module $N\in \Mod S$ satisfying $\Tor_{0\leq i\leq n}^S(C,N)=0$ implies $N=0$ provided that
one of the following conditions is satisfied.
\begin{enumerate}
\item[(1)] $\pd_{R}C\leq n$.
\item[(2)] $_RC$ has an ultimately closed projective resolution at $n$.
\end{enumerate}
\end{corollary}

\begin{proof}
Let $N\in \Mod S$ with $\Tor_{0\leq i\leq n}^S(C,N)=0$. By Proposition \ref{4.8},
we have that $N$ is $C$-adstatic and $N\cong (C\otimes_SN)_*=0$.
\end{proof}

We now are in a position to give the following

\begin{theorem}\label{4.10}
If $_RC$ has an ultimately closed projective resolution, then
$$\mathcal{A}_C(S)={{C_S}^{\top}}={^{\bot}\mathcal{I}_C(S)}.$$
\end{theorem}

\begin{proof}
By the definition of $\mathcal{A}_C(S)$ and Lemma \ref{4.3}, we have
$\mathcal{A}_C(S)\subseteq{{C_S}^{\top}}={^{\bot}\mathcal{I}_C(S)}$.

Now let $N\in{^{\bot}\mathcal{I}_C(S)}$ and let $f:C\otimes_SN\to B$
be a $\mathcal{B}_C(R)$-preenvelope of $C\otimes_SN$ in $\Mod R$ as in Theorem \ref{3.7}.
Because $\mathcal{B}_C(R)$ is injectively coresolving
in $\Mod R$ by \cite[Theorem 6.2]{HW07}, $f$ is monic. By Proposition \ref{4.8},
$\mu_N$ is an isomorphism. Then by Theorem \ref{3.7}(1), we have a monic $\mathcal{A}_C(S)$-preenvelope
$$f^0:N\rightarrowtail A^0$$ of $N$, where $f^0=f_*\mu_N$ and $A^0=B_*$. So we have a
$\Hom_S(-,\mathcal{A}_C(S))$-exact exact sequence
$$0\to N \buildrel {f^0} \over \longrightarrow A^0 \to N^1 \to 0$$
in $\Mod S$, where $N^1=\Coker f^0$. Because $A^0\in{^{\bot}\mathcal{I}_C(S)}$,
we have $N^1\in{^{\bot}\mathcal{I}_C(S)}$. Similar to the above argument, we get a
$\Hom_S(-,\mathcal{A}_C(S))$-exact exact sequence
$$0\to N^1 \buildrel {f^1} \over \longrightarrow A^1 \to N^2 \to 0$$
in $\Mod S$ with $A^1\in\mathcal{A}_C(S)$ and $N^2\in{^{\bot}\mathcal{I}_C(S)}$. Repeating this procedure,
we get a $\Hom_S(-,\mathcal{A}_C(S))$-exact exact sequence
$$0\to N \buildrel {f^0} \over \longrightarrow A^0 \buildrel {f^1} \over \longrightarrow
A^1\buildrel {f^2} \over \longrightarrow \cdots \buildrel {f^i}
\over \longrightarrow A^i\buildrel {f^{i+1}} \over \longrightarrow \cdots$$
in $\Mod S$ with all $A^i$ in $\mathcal{A}_C(S)$. Because $\mathcal{I}_C(S)\subseteq\mathcal{A}_C(S)$
by \cite[Corollary 6.1]{HW07}, this exact sequence is $\Hom_S(-,\mathcal{I}_C(S))$-exact.
By \cite[Theorem 3.11(1)]{TH}, there exists a $\Hom_S(-,\mathcal{A}_C(S))$-exact exact sequence
$$0\to A^i\to U^i_0 \to U^i_1 \to \cdots \to U^i_j\to \cdots$$
in $\Mod S$ with all $U^i_j$ in $\mathcal{I}_C(S)$ for any $i,j\geq 0$. Then by \cite[Corollary 3.9]{Hu},
we get the following $\Hom_S(-,\mathcal{A}_C(S))$-exact exact sequence
$$0\to N\to U^0_0 \to U^0_1\oplus U^1_0 \to \cdots \to \oplus_{i=0}^{n}U^i_{n-i}\to \cdots$$
in $\Mod S$ with all terms in $\mathcal{I}_C(S)$. It follows from \cite[Theorem 3.11(1)]{TH} that
$N\in\mathcal{A}_C(S)$. The proof is finished.
\end{proof}

We use $\pd_{S^{op}}C$ and $\fd_{S^{op}}C$ to denote the projective and flat dimensions of $C_S$ respectively.

\begin{corollary}\label{4.11}
If $_RC$ has an ultimately closed projective resolution,
then the following statements are equivalent for any $n\geq 0$.
\begin{enumerate}
\item[(1)] $\pd_{S^{op}}C\leq n$.
\item[(2)] $\mathcal{A}_{C}(S)$-$\pd_{S}N \leq n$ for any $N\in\Mod S$.
\end{enumerate}
\end{corollary}

\begin{proof}
Assume that $_RC$ has an ultimately closed projective resolution. By Theorem \ref{4.10},
we have $\mathcal{A}_C(S)={{C_S}^{\top}}$. Then
it is easy to see that $C_S$ is flat (equivalently, projective) if and only if $\mathcal{A}_{C}(S)=\Mod S$,
so the assertion for the case $n=0$ follows. Now let $N\in\Mod S$ and $n\geq 1$.

$(2)\Rightarrow (1)$ By (2) and Theorem \ref{4.4}, we have
$\Omega^n(N)\in\mathcal{A}_{C}(S)(\subseteq{C_S}^\top)$. Then by the dimension shifting,
we have $\Tor_{\geq n+1}^S(C,N)=0$, and so $\pd_{S^{op}}C=\fd_{S^{op}}C\leq n$.

$(1)\Rightarrow (2)$ If $\pd_{S^{op}}C\leq n$, then $\Omega^n(N)\in{{C_S}^{\top}}$ by the dimension shifting.
By Theorem \ref{4.10}, we have $\Omega^n(N)\in\mathcal{A}_{C}(S)$ and $\mathcal{A}_{C}(S)$-$\pd_{S}N \leq n$.
\end{proof}

\vspace{0.3cm}

{\bf Acknowledgements.}
This research was partially supported by NSFC (Grant No. 11571164).

\end{document}